%% file: MAIN.tex
\documentclass[12pt,reqno]{amsart}					
\input{preamble.sty}

\title[Two Theorems in Weighted Intermediate Ricci Curvature]{Synge's Theorem and a sphere theorem in Weighted Intermediate Ricci Curvature}

\author{Hasan M. El-Hasan \and Curtis Pro}
\date{\today}
\address{Hasan M. El-Hasan, Department of Mathematics, University of California, Santa Barbara, Santa Barbara, CA 93106}
\email{elhasan@ucsb.edu}
\urladdr{https://www.hmelhasan.com/}

\address{Curtis Pro, Department of Mathematics, California State University, Stanislaus, Turlock, CA 95382}
\email{cpro@csustan.edu}

\begin{document}
\begin{abstract} Using the tools developed by Kennard, Wylie, and Yeroshkin for weighted sectional curvature bounds, we generalize a theorem of Wilhelm to show that Synge's Theorem holds for manifolds with density that have positive weighted intermediate Ricci curvature and a sufficiently large first systole. Furthermore, we provide a sphere theorem for compact manifolds with positive weighted intermediate Ricci curvature.
\end{abstract}

\maketitle

\section{Introduction}
\input{1-intro}

\section{Notation and Preliminary Content}\label{section preliminary}
\input{2-Preliminary}

\section{Generalized Synge's Theorem}\label{section synge}
\input{3-Synge}

\section{Sphere Theorem}\label{section sphere}
\input{4-Sphere}

\section{Example}\label{section example}
\input{5-Example}

\bibliographystyle{plain}
\bibliography{bibliography}

\end{document}

%% file: 1-Intro.tex
Recall that Synge's Theorem \cite{synge} says if $M$ is a closed Riemannian $n$-manifold with positive sectional curvature, then $M$ is simply connected if $n$ is even and $M$ is orientable, and $M$ is orientable if $n$ is odd. In \cite{wilhelm_intermediate_1997}, Wilhelm shows that a Riemannian manifold $M$ with positive Ricci curvature will also satisfy the conclusion of Synge's Theorem provided the first systole $\operatorname{sys}_1 M$, (i.e., the length of the shortest noncontractible loop in $M$) is sufficiently large. 

In fact, Wilhelm shows 
\begin{thm}[Wilhelm, \cite{wilhelm_intermediate_1997}]\label{wilhelm synge theorem} Let $(M^n,g)$ be a complete Riemannian manifold and let $k\in\{1,2,\ldots,n-1\}$ such that $\ric_k \geq k$ and $\sys M>\pi\sqrt{\frac{k-1}{k}}$. If $n$ is even and $M$ is orientable, then $M$ is simply connected. If $n$ is odd, then $M$ is orientable.
\end{thm}

Here, $\ric_k$ is the $k^{\text{th}}$ intermediate Ricci curvature of $M$, and $\ric_k \geq k$ means that for any choice $\{v,e_1,e_2,\ldots, e_k\}$ of orthonormal $(k+1)$-frame, we have 
\[ \sum_{i=1}^k \sec(v, e_i) \geq k. \]
Positive sectional curvature is equivalent to $\ric_1>0$, and positive Ricci curvature is equivalent to $\ric_{n-1}>0$. In general, $\ric_k>0$ is a stronger condition than $\ric_{k+1}>0$. 

Additionally, in \cite{Xia1997}, Xia generalizes the classic homeomorphic sphere theorem to manifolds with positive intermediate Ricci curvature:

\begin{thm}[Xia, \cite{Xia1997}]\label{xia sphere theorem}
    Let $(M^n,g)$ be a complete simply connected Riemannian manifold with $n\geq3$. If $M$ has $\ric_{\floor{\frac{n+1}{2}}}>\floor{\frac{n+1}{2}}$ and there exists a point $p\in M$ such that $\inj_p\geq \frac{\pi}{2}$, then $M$ is homeomorphic to $S^n$.
\end{thm}

In this paper, we generalize Theorems \ref{wilhelm synge theorem} and \ref{xia sphere theorem} to compact manifolds with positive weighted intermediate Ricci curvature, which is defined in Definition \ref{pos wgt int ricci defn} after the following review of weighted curvature.

Weighted curvature is defined for \textbf{smooth metric measure spaces} $(M,g,\mu)$. Here, $(M,g)$ is a Riemannian manifold, and $\mu$ is a smooth measure on $M$. Throughout this paper, we will assume that $\mu = e^{-(n+1)f}\dvol_g$, where $f:M\to\R$ is a smooth function, and we refer to the space $(M,g,\mu)$ as $(M,g,f)$. 

Kennard, Wylie, and Yeroshkin generalize sectional curvature to the weighted setting (see \cites{wylie_weighted_2015,wylie_pinching_2015,WY_weighted_2016,KW_weighted_2017,KWY_weighted_2019}). More precisely, in \cite{KWY_weighted_2019}, weighted sectional curvature is defined as follows:

\begin{defn}\label{defn_of_weighted_sectional}
Let $(M,g)$ be a Riemannian manifold with sectional curvature $\sec$. Let $f:M\to\R$ be a smooth function. Given orthonormal vector fields $U$ and $V$ on $M$, the weighted sectional curvature of the smooth metric measure space $(M,g,f)$ is
    \begin{equation}\label{weighted-sectional}
        \sec^f (U,V) := \sec(U,V) + \hess_f(U,U) + \odif{f}(U)^2.
    \end{equation}
    We say that $(M,g,f)$ has \textbf{positive weighted sectional curvature} if for all orthonormal vector fields $U,V$, we have $\sec^f (U,V) > 0$.
\end{defn}

\begin{remark}
This formulation of weighted sectional curvature was originally denoted by $\overline{\sec}^f$ in \cites{KW_weighted_2017,KWY_weighted_2019,wylie_weighted_2015}, and defined as ``strongly weighted sectional curvature.'' The bar was meant to distinguish from a different formulation that the authors call ``weakly weighted sectional curvature,'' which is not used in this document.
\end{remark}

In \cite{WY_weighted_2016}, Wylie and Yeroshkin introduce the \textbf{weighted Levi--Civita connection}, defined by
\begin{equation}\label{weighted levi civita}
    \nabla^f_X Y := \nabla_X Y - \odif{f}(X)Y - \odif{f}(Y)X,
\end{equation}
which is the unique torsion-free connection with the same geodesics as the ordinary Levi--Civita connection, $\nabla$, up to reparameterization and such that $\mu$ is parallel. Indeed, it induces its own weighted curvature tensor
\begin{equation}\label{weighted curvature tensor}
    R^f(X,Y)Z := \nabla^f_X \nabla^f_Y Z - \nabla^f_Y \nabla^f_X Z -\nabla^f_{[X,Y]} Z,
\end{equation}
and the sectional curvature associated with it is precisely Equation \eqref{weighted-sectional}.

In their papers, Kennard, Wylie, and Yeroshkin include generalizations of Synge's theorem (\cite{KW_weighted_2017}) and the homeomorphic sphere theorem (\cite{wylie_pinching_2015}) for positive weighted sectional curvature. Motivated by recent advances in the study of intermediate and weighted curvatures, weighted intermediate Ricci curvature has emerged as a new topic of interest in the literature (\cites{lee-wei-2026,lin-wang-2026,ralphxu}). Positive weighted intermediate Ricci curvature is defined as follows:

\begin{defn}\label{pos wgt int ricci defn}
    A smooth metric measure space $(M,g,f)$ has \textbf{positive weighted k$^{\text{th}}$-intermediate Ricci curvature}, denoted $\ric^f_k > 0$, if for all unit vector fields $U$ and all orthonormal sets $ \{V_1,\cdots,V_k\} $ in $ U^\perp $, we have
    \[ \sum_{i=1}^k \sec^f (U,V_i) > 0. \]
\end{defn}

Our generalizations of Theorems \ref{wilhelm synge theorem} and \ref{xia sphere theorem} are as follows:

\begin{thmalph}\label{main thm A}
    Let $(M^n,g,f)$ be a closed smooth metric measure space, and let $\kappa>0$ and $k\in\{1,2,\ldots,n-1\}$ such that
    \begin{equation*}
    \ric^f_k \geq \kappa e^{-4f}
    \qquad\text{and}\qquad
    \sysf(M) >\pi\sqrt{\frac{k-1}{\kappa}}.  
    \end{equation*}
    If $n$ is even and $M$ is orientable, then $M$ is simply connected. If $n$ is odd, then $M$ is orientable.
\end{thmalph} 

\begin{thmalph}\label{main thm B}
    Let $(M^n,g,f)$ be a closed, simply-connected smooth metric measure space with $n\geq4$. If $M$ satisfies 
    \[\ric^f_{\floor{\frac{n+1}{2}}}>\floor{\frac{n+1}{2}}e^{-4f}\]
    and there exists a point $p\in M$ such that 
    \[\injf_p\geq \frac{\pi}{2},\]
    then $M$ is homeomorphic to $S^n$.
\end{thmalph}

The quantity $\floor{x}$ rounds down a real number $x$ to the nearest integer. The quantities $\sysf M$ and $\injf_p$ are reparameterized versions of the first systole and injectivity radius, respectively. They are formally defined in Subsections \ref{weighted first systole subsec} and \ref{weighted inj rad subsec}.

Note that both Theorems \ref{main thm A} and \ref{main thm B} require $M$ to be closed, rather than just complete. This is to ensure that $f$ is bounded, as well as the existence of the minimizing geodesics and uniform curvature bounds. 

In Section \ref{section preliminary}, we review the weighted versions of common geometric concepts including geodesics, the second fundamental form, and variations of energy. Sections \ref{section synge} and \ref{section sphere} contain the proofs to Theorems \ref{main thm A} and \ref{main thm B}, respectively. In \ref{section example}, we provide an example of a rotationally symmetric metric on $S^n$ that has positive weighted $2^{\text{nd}}$-intermediate Ricci curvature but not non-negative $2^{\text{nd}}$-intermediate Ricci curvature.

\begin{custom}{Acknowledgment}
    The authors are grateful to Will Wylie and Lee Kennard for their advice during the early stages of this project.
\end{custom}

\begin{custom}{AI Disclaimer}
    ChatGPT-6 Astra was used only to proofread the final version of this document.
\end{custom}

%% file: 2-Preliminary.tex
\subsection{Geodesic Reparameterization and Conformal Change Metrics}\label{reparam subsec}
Following with the notation introduced in \cite{KWY_weighted_2019}, to differentiate between geodesics with respect to either connection, we say that a geodesic of $\nabla$ is a $\boldsymbol{g}$\textbf{-geodesic}, and a geodesic of $\nabla^f$ is an $\boldsymbol{f}$\textbf{-geodesic}. Since $\nabla^f$ is projectively equivalent to $\nabla$, the geodesics of either connection are the same up to reparameterization. If $c:[a,b]\rightarrow M$ is a unit speed $g$-geodesic parameterized by arc-length, then $\gamma(s(t)):=c(t)$, where 
\begin{equation}\label{s parameter}
    s(t)=\int^t_a e^{-2f(c(\tau))}\odif{\tau}
\end{equation}
is its corresponding $f$-geodesic and is said to have \textbf{standard parameterization}. We write $c'(t)$ to denote the derivative with respect to $t$ of a $g$-geodesic $c(t)$, and $\dot\gamma(s)$ to denote the derivative with respect to $s$ of a $f$-geodesic $\gamma(s)$. In particular, the two parameterizations are related as follows:
\begin{equation}\label{gamma dot c prime equation}
    \dot\gamma = e^{2f}c'.
\end{equation}

Define $\len_g(c|_{[0,a]})$ to be the $\boldsymbol{g}$\textbf{-length} from $t=0$ to $t=a$ of a path $c(t)$ over $M$.
Notice that if $h = e^{-4f}g$ is a conformal change metric of $g$, then the standard parameter length of an $f$-geodesic equals the $h$-length of its underlying $g$-geodesic because 
\begin{align}\len_h(c|_{[0,a]}) &= \int_0^a |c'(t)|_h\dt\\ &= \int_0^ae^{-2f(c(t))}|c'(t)|_g\dt\nonumber\\
&= \int_0^a e^{-2f(c(t))}\dt,\nonumber
\end{align}
which is the length of the $f$-geodesic corresponding to $c(t)$, parameterized on the interval $[0,s(a)]$.

\subsection{Variations of Energy}\label{energy subsec}
Let $c:[0,l]\to M$ be a $g$-geodesic, and let $\gamma:[0,L]\to M$ be the associated $f$-geodesic with standard parameterization $s$. Suppose $U$ is a vector field such that $U\perp\dot\gamma$ along $\gamma$. Proposition 4.1 and Remark 4.2 of \cite{KWY_weighted_2019} provide the index form written in terms of the $f$-geodesic $\gamma$:

\begin{equation}\label{weighted index form}
        I(e^fU,e^fU) = \int_0^L 
        \left(\abs{\nabla_{\dot\gamma}U}^2 - g(R^f(U,\dot\gamma)\dot\gamma, U)\right)\ds +
        \left.\odif{f}(\dot\gamma)\abs{U}^2\right|_0^L.
    \end{equation}

Let $C:(-\epsilon,\epsilon)\times[0,l]\to M$ be a smooth variation of $c$ satisfying $\left. \odv[order=1]{C}{r} \right|_{r=0}=e^{f}U$, where $r$ is the varying parameter. Then the second variation of energy for $C$ is
\begin{equation}\label{weighted energy variation}
\begin{split}
    \left.\odv[order=2]{\cal{E}(C)}{r}\right|_{r=0}  
 = & \int_0^L 
        \left(\abs{\nabla_{\dot\gamma}U}^2 - g(R^f(U,\dot\gamma)\dot\gamma, U)\right)\ds \\
        & + \left.\odif{f}(\dot\gamma)\abs{U}^2\right|_{s=0}^{s=L}
        + \left.g\left( \odv[order=2]{C}{r}(0,t), c'(t) \right)\right|_{t=0}^{t=l}
\end{split}
\end{equation}

\subsection{Second Fundamental Form}
Let $H$ be a codimension 1 hypersurface of $(M^n,g,f)$ with normal vector $N$, and let $X,Y\in TH$. The \textbf{weighted shape operator} is 
\begin{equation}\label{shape operator a}
    S^f(X) = (\nabla^f_X N)^\mathbf{T}.
\end{equation}
where $^\mathbf{T}$ refers to the tangential component of the underlying vector field. Then
\begin{equation}\label{shape operator b}
    S^f(X) = S(X) -\odif{f}(N)X
\end{equation}
since 
\[\nabla^f_X N = \nabla_X N -\odif{f}(X)N -\odif{f}(N)X \]
and the tangential component of $\odif{f}(X)N$ is zero. Thus, the \textbf{weighted second fundamental form} is defined to be
\begin{equation}\label{weighted II}
    \II^f(X,Y) = \II(X,Y) -\odif{f}(N)g(X,Y).
\end{equation}

For some $k\in\{1,\cdots,n-1\}$, $H$ is said to be $\mathbf{k}$\textbf{-convex} if for every $q\in H$ and every orthonormal $k$-frame $\{e_1,\cdots,e_k\}\in T_qH$,
\begin{equation}\label{k-convex}
    \sum_{i=1}^k \II(e_i,e_i) >0.
\end{equation}
$H$ is said to be \textbf{weighted }$\mathbf{k}$\textbf{-convex} if the above inequality also holds for $\II^f$.

\subsection{Reparameterized First Systole}\label{weighted first systole subsec}
Let $\cal C$ be the set of all noncontractible closed curves on $M$.
We define the {reparameterized first systole} of $(M,g,f)$ to be
\begin{equation}\label{weighted first systole defn}
\begin{split}
    \sysf M & = \inf_{c\in\cal C}\left\{ \int_{c(t)} \abs{c'(t)}_g e^{-2f(c(t))}\dt \right\}\\
    & = \sys(M,h),
\end{split}
\end{equation}
where $h$ is the conformally warped metric $h=e^{-4f}g$.

\subsection{Reparameterized Injectivity Radius}\label{weighted inj rad subsec}
For any $(p,v)\in TM$, let $c_v(t)=\exp_p(tv)$. We define the {reparameterized injectivity radius} of $(M,g,f)$ at $p$ to be

\begin{equation}\label{weighted inj rad defn}
    \injf_p = \sup_{r\in(0,\inj_p)}\inf_{\{v\in T_p M \,:\, \abs{v}=1\}}\left\{
    \len_h\left(\left.c_v\right|_{[0,r]}\right)
    \right\},
\end{equation}
where $h$ is the conformally warped metric $h=e^{-4f}g$.


%% file: 3-Synge.tex
\begin{defn}
    For a smooth metric measure space $(M,g,f)$, define 
    \[\kappa(f,k) := \inf_M \left(e^{4f}\ric^f_k\right)\] to be the infimum of $e^{4f}\ric^f_k$ over all points and orthonormal frames.
\end{defn}

Theorem \ref{main thm A} is restated below:

\begin{thm}
    Let $(M^n,g,f)$ be a closed smooth metric measure space, and let $\kappa>0$ and $k\in\{1,2,\ldots,n-1\}$ such that
    \begin{equation*}
    \ric^f_k \geq \kappa e^{-4f}
    \qquad\text{and}\qquad
    \sysf(M) >\pi\sqrt{\frac{k-1}{\kappa}}.  
    \end{equation*}
    If $n$ is even and $M$ is orientable, then $M$ is simply connected. If $n$ is odd, then $M$ is orientable.
\end{thm} 

Our proof of Theorem \ref{main thm A} is split between the following two propositions.

\begin{prop}\label{synge prop 1}
    Let $(M^n,g,f)$ be a closed smooth metric measure space, and let $c:[0,l]\to M$ be a noncontractible, closed $g$-geodesic that is parameterized by arc-length and locally minimizes length with respect to $g$ on the free loop space of $M$. Suppose that parallel translation of the normal space $(c')^{\perp}$ around $c$ fixes a unit vector. Then for all $k\in\{1,2,\cdots,n-1\}$, \[\kappa(f,k)L^2\leq (k-1)\pi^2,\]
    where 
    \[L= \len_h(c) =\int_0^l e^{-2f(c(t))}\dt,\]
    and $h=e^{-4f}g$. 
\end{prop}
\begin{proof}
For brevity, we write $\kappa=\kappa(f,k)$, so that $\ric^f_k\geq \kappa e^{-4f}$. Let $\gamma:[0,L]\to M$	be the corresponding $f$-geodesic to $c$ with standard parameterization. Extend both $c$ and $\gamma$ periodically to smooth maps $c:\R\to M$ and $\gamma:\R\to M$, with periods $l$ and $L$, respectively. By our hypothesis, there exists a unit vector field $E(t)\in T_{c(t)}M$ that is parallel and orthogonal to $c'(t)$, with period $l$. Let 
\[C(a,t)=\exp_{c(t)}(a e^{f(c(t))} E(t))\]
be a smooth variation of $c(t)$. Since $c$ locally minimizes length in the free loop space of $M$, we have by \eqref{weighted energy variation} that
\begin{equation}\label{new synge second var energy}
    \begin{split}
    0 \leq \left.\odv[order=2]{\cal{E}(C)}{a}\right|_{a=0} 
    & = \int_0^{L} 
    \left(\abs{\nabla_{\dot\gamma}E}^2 - g(R^f(E,\dot\gamma)\dot\gamma, E)\right)\ds\\
    & =\int_0^{L} - g(R^f(E,\dot\gamma)\dot\gamma, E)\,\ds.
    \end{split}
\end{equation}

The boundary terms 
\[\left.df(\dot\gamma)|E|^2\right|_{s=0}^{s=L}
\qquad\text{and}\qquad
\left.g\left( \odv[order=2]{C}{a}(0,t), c'(t) \right)\right|_{t=0}^{t=l}
\]
from \eqref{weighted energy variation} are identically zero due to periodicity, and thus are absent in \eqref{new synge second var energy}. The integrand is also $L$-periodic, so for all $r\in\R$,
\begin{equation}
\int_r^{r+L}
g\left(R^f(E,\dot\gamma)\dot\gamma,E\right)\ds
\leq 0.
\label{new-synge-periodic-field}
\end{equation}

Fix $r\in\mathbb R$, and let $\{U_i(t)\}_{i=1}^{k-1}$  be any orthonormal $(k-1)$-frame of parallel vector fields along $c$ in $c'^\perp \cap E^\perp$. Since $\dot\gamma=e^{2f}c'$ and $\ric^f_k\geq\kappa e^{-4f}$, we have
\[
g\left(R^f(E,\dot\gamma)\dot\gamma,E\right)
+ \sum_{i=1}^{k-1}
g\left(R^f(U_i,\dot\gamma)\dot\gamma,U_i\right)
\geq\kappa.
\]
Then integrating the previous display over $[r,r+L]$ yields
\begin{equation}
\sum_{i=1}^{k-1}\int_r^{r+L}
g\left(R^f(U_i,\dot\gamma)\dot\gamma,U_i\right)\ds
\geq\kappa L,
\label{new-synge-curvature-integral}
\end{equation}
where the absence of the term $g\left(R^f(E,\dot\gamma)\dot\gamma,E\right)$ is due to \eqref{new-synge-periodic-field}. Along $\gamma|_{[r,r+L]}$, define
\[
V_i(s)=\sin\left(\frac{\pi(s-r)}L\right)U_i(s),
\]
which vanish at both endpoints. Due to the local length minimality of $c$, we have by \eqref{weighted index form} that
\begin{align*}
0
&\le\sum_{i=1}^{k-1}I\left(e^{f(\gamma(s))}V_i(s),e^{f(\gamma(s))}V_i(s)\right)\\
&=\sum_{i=1}^{k-1}\int_r^{r+L}
\left(
|\nabla_{\dot\gamma}V_i|^2
-g\left(R^f(V_i,\dot\gamma)\dot\gamma,V_i\right)
\right)\ds\\
&=\sum_{i=1}^{k-1}\int_r^{r+L}
\left[
\frac{\pi^2}{L^2}\cos^2\left(\frac{\pi(s-r)}L\right)
-\sin^2\left(\frac{\pi(s-r)}L\right)
g\left(R^f(U_i,\dot\gamma)\dot\gamma,U_i\right)
\right]\ds\\
&=\frac{(k-1)\pi^2}{2L}
-\sum_{i=1}^{k-1}\int_r^{r+L}
\sin^2\left(\frac{\pi(s-r)}L\right)
g\left(R^f(U_i,\dot\gamma)\dot\gamma,U_i\right)\ds.
\end{align*}
In particular, 
\begin{equation}
\sum_{i=1}^{k-1}\int_r^{r+L}
\sin^2\left(\frac{\pi(s-r)}L\right)
g\left(R^f(U_i,\dot\gamma)\dot\gamma,U_i\right)\ds
\le\frac{(k-1)\pi^2}{2L}
\label{new-synge-sine-estimate}
\end{equation}
for all $r\in \R$ and every such parallel frame. 

We claim that there exists a choice of $r$ and $\{U_i\}_{i=1}^{k-1}$ so that 
\begin{equation}
\sum_{i=1}^{k-1}\int_r^{r+L}
\cos^2\left(\frac{\pi(s-r)}L\right)
g\left(R^f(U_i,\dot\gamma)\dot\gamma,U_i\right)\ds
\le\frac{(k-1)\pi^2}{2L}.
\label{new-synge-cosine-estimate}
\end{equation}
Adding \eqref{new-synge-sine-estimate} and \eqref{new-synge-cosine-estimate}
yields
\[\sum_{i=1}^{k-1}\int_r^{r+L}
g\left(R^f(U_i,\dot\gamma)\dot\gamma,U_i\right)\ds
\leq\frac{(k-1)\pi^2}{L},\]
and by \eqref{new-synge-curvature-integral}, we have
\[
\kappa L \leq \frac{(k-1)\pi^2}{L}.
\]
Our desired inequality is the result of multiplying $L$ to both sides.

It remains to show \eqref{new-synge-cosine-estimate}. Define parallel transport along $\gamma$ by
\[
P_a^b:T_{\gamma(a)}M\to T_{\gamma(b)}M.
\] 
Let $\mathcal F$ be the bundle over $[0,L]$ whose fiber at $r$ are the orthonormal
$(k-1)$-frames in $\dot\gamma(r)^\perp\cap E(r)^\perp$. On $\cal F$, define 
\[ \mathcal I(r,\{u_i\}):= 
\sum_{i=1}^{k-1}\int_r^{r+L/2}
\cos^2\left(\frac{\pi(s-r)}L\right)
g\left( R^f(P_r^su_i,\dot\gamma)\dot\gamma, P_r^su_i \right)\ds.
\]
Since $\mathcal F$ is compact, we can define $(r_0,\{u_i\})$ to be where $\mathcal I$ attains its minimal value.
Then we choose
\[
U_i(s)=P_{r_0}^su_i,
\]
and by the periodicity of $\gamma$ and $E$, 
\[
\left\{P_{r_0}^{r_0+L}u_i\right\}_{i=1}^{k-1}
\]
is also an orthonormal frame in
$\dot\gamma(r_0)^\perp\cap E(r_0)^\perp$. Therefore,
\[
\mathcal I(r_0,\{u_i\})
\leq \mathcal I\left( r_0,\left\{P_{r_0}^{r_0+L}u_i\right\}
\right).
\]
Expanding, this inequality becomes
\begin{equation}
\begin{aligned}
&\sum_{i=1}^{k-1}\int_{r_0}^{r_0+L/2}
\cos^2\left(\frac{\pi(s-r_0)}L\right)
g\left(R^f(U_i,\dot\gamma)\dot\gamma,U_i\right)\ds\\
&\qquad\le
\sum_{i=1}^{k-1}\int_{r_0+L}^{r_0+3L/2}
\cos^2\left(\frac{\pi(s-r_0)}L\right)
g\left(R^f(U_i,\dot\gamma)\dot\gamma,U_i\right)\ds.
\end{aligned}
\label{new-synge-half-interval}
\end{equation}
Adding to both sides the common integral over $[r_0+L/2,r_0+L]$ yields
\begin{equation}
\begin{aligned}
&\sum_{i=1}^{k-1}\int_{r_0}^{r_0+L}
\cos^2\left(\frac{\pi(s-r_0)}L\right)
g\left(R^f(U_i,\dot\gamma)\dot\gamma,U_i\right)\ds\\
&\qquad\le
\sum_{i=1}^{k-1}\int_{r_0+L/2}^{r_0+3L/2}
\cos^2\left(\frac{\pi(s-r_0)}L\right)
g\left(R^f(U_i,\dot\gamma)\dot\gamma,U_i\right)\ds.
\end{aligned}
\label{new-synge-interval-comparison}
\end{equation}

Lastly, we apply \eqref{new-synge-sine-estimate} on the interval $[r_0+L/2,r_0+3L/2]$, using our choice of parallel fields $U_i$. Since
\[
\sin^2\left(\frac{\pi(s-(r_0+L/2))}L\right)
=\cos^2\left(\frac{\pi(s-r_0)}L\right),
\]
we have that
\[
\sum_{i=1}^{k-1}\int_{r_0+L/2}^{r_0+3L/2}
\cos^2\left(\frac{\pi(s-r_0)}L\right)
g\left(R^f(U_i,\dot\gamma)\dot\gamma,U_i\right)\ds
\le\frac{(k-1)\pi^2}{2L}.
\]
Combining the above display with \eqref{new-synge-interval-comparison} proves \eqref{new-synge-cosine-estimate}, which completes the proof.
\end{proof}


\begin{prop}\label{synge prop 2}
Let $(M^n,g,f)$ be a closed smooth metric measure space, and let $k\in\{1,\cdots,n-1\}$. Suppose
\[
\ric^f_k\geq \kappa e^{-4f}
\]
for some $\kappa>0$. For each nontrivial free homotopy class $[\alpha]$ of closed curves in $M$, let $c_\alpha:[0,l_\alpha]\to M$ be a closed $g$-geodesic parametrized by arc length which minimizes length in $[\alpha]$. Let $\gamma_\alpha:[0,L_\alpha]\to M$ be its corresponding $f$-geodesic with standard parametrization such that
\[
L_\alpha :=\len_h(c_\alpha)
> \pi\sqrt{\frac{k-1}{\kappa}},
\]
where $h=e^{-4f}g$.
If $n$ is even and $M$ is orientable, then $M$ is simply connected. If $n$ is odd, then $M$ is orientable.
\end{prop}

\begin{proof}
Towards a contradiction, suppose that $n$ is even and $M$ is orientable but not simply connected, or $n$ is odd and $M$ is nonorientable. If $n$ is even, let $[\alpha]$ be any nontrivial free homotopy class, which must preserve orientation since $M$ is assumed to be orientable. If $n$ is odd, then $M$ being nonorientable implies the existence of an orientation-reversing loop, and we choose $[\alpha]$ to be its free homotopy class. $[\alpha]$ is then necessarily nontrivial. In both cases, it suffices to show that $L_\alpha \leq \pi\sqrt{\frac{k-1}{\kappa}}$, a contradiction.

Since $c_\alpha$ minimizes length in $[\alpha]$, it also locally minimizes length on the free loop space of $M$. A general result of linear algebra states that every orthogonal linear transformation of $\R^{n-1}$ with $\det A = (-1)^{n}$ fixes a nonzero vector of $\R^{n-1}$. Let $P:c_\alpha'(0)^\perp\to c_\alpha'(0)^\perp$ denote parallel transport around $c_\alpha$, restricted to vectors in $c_\alpha'(0)^\perp$. This map satisfies 
\[\det P=(-1)^n,\]
so $P$ fixes a nonzero vector in $c_\alpha'(0)^\perp$. Then $c_\alpha$ satisfies the hypothesis of Proposition \ref{synge prop 1}, and so
\[
\kappa L_\alpha^2
\leq (k-1)\pi^2.
\]
Dividing by $\kappa>0$ and taking square roots proves the result.
\end{proof}

Theorem \ref{main thm A} follows as a corollary of Proposition \ref{synge prop 2}, since $\ric^f_k\geq\kappa e^{-4f}$ and $\sysf M>\pi\sqrt{\frac{k-1}{\kappa}}$ implies the hypothesis of Proposition \ref{synge prop 2}, since
\[L_\alpha \geq \sysf M.\]

%% file: 4-Sphere.tex
Theorem \ref{main thm B} is restated below:
\begin{thm}
    Let $(M^n,g,f)$ be a closed, simply-connected smooth metric measure space with $n\geq4$. If $M$ satisfies 
    \[\ric^f_{\floor{\frac{n+1}{2}}}>\floor{\frac{n+1}{2}}e^{-4f}\]
    and there exists a point $p\in M$ such that 
    \[\injf_p\geq \frac{\pi}{2},\]
    then $M$ is homeomorphic to $S^n$.
\end{thm}

The proof of Theorem \ref{main thm B} relies on the following result of Xu (Theorem 1.2, \cite{ralphxu}), which generalizes Theorem 1 in \cite{Wu1987} to manifolds with positive weighted intermediate Ricci curvature:

\begin{lemma}[Xu, \cite{ralphxu}]\label{Wylie-Xu lemma}
Let $(M^n,g,f)$ be a compact smooth metric measure space with nonempty boundary $\partial M$ and let $k\in\{1,\cdots,n-1\}$. If $M$ has $\ric^f_k>0$ and $\partial M$ is weighted $k$-convex, then $M$ has the homotopy type of a CW complex obtained from $\partial M$ by attaching a finite number of cells each of dimension $\geq n-k+1$.
\end{lemma}

\begin{proof}[Proof of Theorem \ref{main thm B}]
For brevity, we write $k=\floor{\frac{n+1}{2}}$. Since $M$ is compact, there exists a $\delta>0$ such that 
\[\ric^f_{k} > \left(k+\delta\right)e^{-4f}.\]

For any $l\in(0,\inj_p)$, the open $g$-geodesic ball $B(p,l)$ centered at $p$ with radius $l$ is homeomorphic to the interior of the disc $D^n$ and $\partial \overline{B(p,l)}$ is homeomorphic to $S^{n-1}$. In particular, we choose an $l$ such that

\begin{equation}\label{sphere l bound}
    \inf_{\{v\in T_p M \,:\, \abs{v}=1\}}
    \left\{\int_0^l e^{-2f(\exp_p(vt))}\dt\right\} >\frac{\pi}{2}\sqrt{\frac{k}{k+\delta}}.
\end{equation}

To show that $M$ is homeomorphic to $S^n$, we must fit $M$ into the hypothesis of Lemma $\ref{Wylie-Xu lemma}$, therefore it suffices to show that $M\setminus B(p,l)$ has weighted $k$-convex boundary. For any $q\in\partial \overline{B(p,l)}$, let $c:[0,l]\to M$ be a normal $g$-geodesic such that $c(0)=p$, $c(l)=q$, and $c'(l)=-N$, where $N$ is the outward unit normal vector of $M\setminus B(p,l)$. Let $\gamma:[0,L]\to M$ be its corresponding $f$-geodesic with standard parameterization, where 

\begin{equation}\label{sphere L bound}
    L=\int_0^l e^{-2f(c(t))}\dt > \frac{\pi}{2}\sqrt{\frac{k}{k+\delta}}.
\end{equation}
Then $\dot{\gamma}(L)=e^{2f(q)}(-N)$. Notice that the value of $L$ depends on the choice of $q$, and that the following computations are meant to be realized for all geodesics from $p$ to each $q\in\partial \overline{B(p,l)}$. Each $L$ shares the qualities described in \eqref{sphere L bound}.

Choose an orthonormal $k$-frame $\{u_1,\cdots,u_k\}\subset T_q\partial \overline{B(p,l)}$ and let $\{U_1(t),\cdots,U_k(t)\}$ be their parallel translation backwards along $c(t)$. Then for each $i\in\{1,\cdots,k\}$, we have
\[U_i(l)=u_i,\qquad \nabla_{c'}U_i=0, \qquad\text{and}\qquad U_i(t)\perp c'(t). \]
Along $\gamma$, we define the following two fields:
\begin{equation}\label{sphere V field}
    {V}_i(s) = \sin\left(\frac{\pi s}{2L}\right){\oU}_i(s),
\end{equation}
and
\begin{equation}\label{sphere W field}
\begin{split}
    {W}_i(s) & = e^{f(\gamma(s))}{V}_i(s) \\
    & = e^{f(\gamma(s))}\sin\left(\frac{\pi s}{2L}\right){\oU}_i(s).
\end{split}
\end{equation}
Thus we have
\[{V}_i(0)={W}_i(0)=0, \qquad{V}_i(L)=u_i, \qquad\text{and}\qquad{W}_i(L)=e^{f(q)}u_i.\]
For each $i\in\{1,\cdots,k\}$, choose a smooth variation $C_i:(-\epsilon,\epsilon)\times[0,l]\to M$ that satisfies
\[C_i(0,t)=c(t), \qquad 
C_i(r,0)=p, \qquad
C_i(r,l)\in\partial \overline{B(p,l)},
\]
and
\[
\left. \odv[order=1]{C_i}{r} \right|_{r=0}=W_i(t). \]

Since $c$ is a minimizing geodesic between $p$ and $q$, Equation \eqref{weighted energy variation} gives us

\begin{equation}\label{weighted energy variation 2}
\begin{split}
   0\leq & \int_0^L 
        \left(\abs{\nabla_{\dot\gamma}V_i}^2 - g\left(R^f(V_i,\dot\gamma)\dot\gamma, V_i\right)\right)\ds \\
        & \; + \left. \odif{f} \left(\dot{\gamma}(s)\right) |V_i(s)|^2\right|_{s=0}^{s=L}
        + \left.g\left( \odv[order=2]{C_i}{r}(0,t), c'(t) \right)\right|_{t=0}^{t=l}.
\end{split}
\end{equation}

The first boundary term in (\ref{weighted energy variation 2}) becomes
\begin{align}
    \left.\odif{f} \left(\dot{\gamma}(s)\right) |V_i(s)|^2\right|_{s=0}^{s=L}
    = & \odif{f}\left(\dot{\gamma}(L)\right) g\left(V_i(L),V_i(L)\right) \nonumber\\
    = & e^{2f(q)}\odif{f}(-N)g(u_i,u_i)\nonumber\\
    = & -e^{2f(q)}\odif{f}(N).\label{sphere g bdy condition}
\end{align}

The second boundary term in (\ref{weighted energy variation 2}) becomes 

\begin{align}
    \left.g\left( \odv[order=2]{C_i}{r}(0,t), c'(t) \right)\right|_{t=0}^{t=l} 
    = & g\left(\nabla_{W_{i}(l)} W_{i}(l), c'(l)\right) \nonumber\\
    = & e^{2f(q)}g(\nabla_{u_i} u_i,-N) \nonumber\\
    = & e^{2f(q)}g(\nabla_{u_i} N, u_i) \nonumber\\
    = & e^{2f(q)}\II(u_i,u_i). \label{sphere f bdy condition}
\end{align}
Thus

\begin{align*}
    \left.\odif{f}\left(\dot{\gamma}(s)\right) |V_i(s)|^2\right|_{s=0}^{s=L}
        + & \left.g\left( \odv[order=2]{C_i}{r}(0,t), c'(t) \right)\right|_{t=0}^{t=l} \\
    & \qquad = -e^{2f(q)}\odif{f}(N) + e^{2f(q)}\II(u_i,u_i) \\
    & \qquad =  e^{2f(q)}\II^f(u_i,u_i),
\end{align*}
and so (\ref{weighted energy variation 2}) becomes 
\begin{equation}\label{weighted energy variation 3}
    e^{2f(q)}\II^f(u_i,u_i) \geq -\int_0^L 
        \left(\abs{\nabla_{\dot\gamma}V_i}^2 - g(R^f(V_i,\dot\gamma)\dot\gamma, V_i)\right)\ds.
\end{equation}

Summing over $i$ yields

\begin{align*}
    e^{2f(q)}\sum_{i=1}^k \II^f(u_i,u_i) 
    \geq & - \int_0^L 
    \frac{\pi^2 k}{4L^2}\cos^2\left(\frac{\pi s}{2L}\right) \\
    &\qquad
    - e^{4f(\gamma(s))}\sin^2\left(\frac{\pi s}{2L}\right) 
        \sum_{i=1}^k g\left( R^f(U_i,c')c',U_i \right) 
    \ds \\
    > & - \int_0^L 
    \frac{\pi^2k}{4L^2}\cos^2\left(\frac{\pi s}{2L}\right)
    - (k+\delta) \sin^2\left(\frac{\pi s}{2L}\right) 
    \ds \\
    = & \; \frac{4(k+\delta)L^2-\pi^2 k}{8L} \\
    > & \; 0,
\end{align*}
where the last inequality is due to \eqref{sphere L bound}. Therefore $\sum_{i=1}^k \II^f(u_i,u_i)>0$, meaning that $\partial \overline{B(p,l)}$ is weighted $k$-convex. By Lemma \ref{Wylie-Xu lemma}, $M\setminus B(p,l)$ has the homotopy type of a CW complex obtained by attaching cells of dimension at least $\floor{n/2} + 1$ to $S^{n-1}$. This gives that $\pi_i(M\setminus B(p,l)) =0$ for all $i\in \{1,\cdots,\floor{n/2}\}$ and since $M$ is obtained from $M\setminus B(p,l)$ by attaching the $n$-cell $\overline {B(p,l)}$ to $S^{n-1}$, we have $\pi_i(M) = 0$ for all $i\in\{1,\cdots,\floor{n/2}\}$. By Hurewicz's Theorem, Poincar\'e Duality, and the assumption $M$ is orientable, we get $H_0(M,\mathbb Z) = H_n(M,\mathbb Z) = \mathbb Z$, and $H_i(M,\mathbb Z) = 0$ for all $i\in\{1,\cdots,n-1\}$. Therefore, $M$ is a homology sphere and by the Generalized Poincar\'e theorem, $M$ is homeomorphic to $S^n$. 
\end{proof}




%% file: 5-Example.tex
Proposition 3.11 in \cite{KW_weighted_2017} constructs a rotationally symmetric metric on $S^n$ with positive weighted sectional curvature but not nonnegative sectional curvature. To illustrate how Theorem \ref{main thm A} generalizes Theorem \ref{wilhelm synge theorem}, we give an example of a metric on $S^n$ that has positive weighted $2^{\text{nd}}$-intermediate Ricci curvature but not non-negative $2^{\text{nd}}$-intermediate Ricci curvature.  

\begin{lemma}\label{example verify lemma}
    Let $n\geq 4$, and $(M,g,f)$ be an $n$-dimensional smooth metric measure space such that $g=\dt^2 + \phi^2(t)g_{S^{n-1}}$ is a warped product metric and $f,\phi:[0,\pi]\rightarrow \mathbb{R}$ are smooth functions that depend only on $t$. Then $M$ will satisfy $\ric_2^f>\lambda$ if and only if
    \begin{align*}
        \lambda \;<\;\; & 
        2\left(f''+(f')^2 
        - \frac{\phi''}{\phi}\right), \\[1ex]
        \lambda \;<\;\; & 
        \frac{1-(\phi')^2}{\phi^2}
        + \frac{2f'\phi' - \phi''}{\phi}, \\[1ex]
        \lambda \;<\;\; & 2\left( \frac{1-(\phi')^2}{\phi^2}
        + \frac{f'\phi'}{\phi} \right).
    \end{align*}
\end{lemma}
\begin{proof}
    Lemma 3.5 in \cite{KW_weighted_2017} establishes that $M$ satisfies $\sec^f>\lambda$ if and only if
    \begin{align}
    \lambda \;<\; \sec^f(\partial_t,Y)&=-\frac{\phi''}{\phi}+f''+(f')^2, \label{example verify 1}\\
    \lambda \;<\; \sec^f(Y,\partial_t)&=-\frac{\phi''}{\phi}+f'\frac{\phi'}{\phi}, \label{example verify 2}\\
    \lambda \;<\; \sec^f(Y,Z)&=\frac{1-(\phi')^2}{\phi^2}+f'\frac{\phi'}{\phi},\label{example verify 3}
\end{align}
    for all orthonormal pairs $Y,Z$ tangent to $S^{n-1}$. Each inequality in the statement of the lemma is derived by taking the following sums:
    \begin{align*}
    \sec^f(\partial_t,Y) + \sec^f(\partial_t,Z), \\
    \sec^f(Y,\partial_t) + \sec^f(Y,Z), \\
    \sec^f(Y,Z) + \sec^f(Y,W), \\
    \end{align*}
    where $W$ is also tangent to $S^{n-1}$ and $\{Y,Z,W\}$ is an orthonormal set.
\end{proof}

\begin{prop}\label{the example}
    For every $n\geq4$, $S^n$ admits a rotationally symmetric metric $g$ that supports a function $f$ such that $\ric^f_2>0$, but fails to satisfy $\ric_2\geq 0$.
\end{prop}
\begin{proof}
    For $t\in[0,\pi]$, let $g=\dt^2 + \phi^2(t)g_{S^{n-1}}$ be a metric for $S^n$, where 
    \begin{equation}\label{example phi}
    \phi(t) = \sin t\left( 1+ \frac{1}{4}\sin^2 t \right),
    \end{equation}
    and let
    \begin{equation}\label{example density}
        f(t) = \frac{1}{2}\sin^2 t.
    \end{equation}
    Then 
    \begin{align*}
    2\left(f''+(f')^2-\frac{\phi''}{\phi}\right)
    &= 2\left(
    1-2\sin^2 t+\sin^2 t\cos^2 t
    -\frac{2-9\sin^2 t}{4+\sin^2 t}
    \right) \\
    &= \frac{2\left(2+\sin^2 t\cos^2 t(6+\sin^2 t)\right)}
    {4+\sin^2 t}
    >0,
    \end{align*}
    and
    \begin{align*}
    \frac{1-(\phi')^2}{\phi^2}
    +\frac{2f'\phi'-\phi''}{\phi}
    &= \frac{-8+15\sin^2 t+9\sin^4 t}{(4+\sin^2 t)^2} \\
    &\quad+\frac{2\cos^2 t(4+3\sin^2 t)-(2-9\sin^2 t)}
    {4+\sin^2 t} \\
    &= \frac{16+35\sin^2 t
    +2\sin^2 t\cos^2 t(7+3\sin^2 t)}
    {(4+\sin^2 t)^2}
    >0,
    \end{align*}
    and
    \begin{align*}
    2\left(
    \frac{1-(\phi')^2}{\phi^2}
    +\frac{f'\phi'}{\phi}
    \right)
    &= 2\left(
    \frac{-8+15\sin^2 t+9\sin^4 t}{(4+\sin^2 t)^2}
    +\frac{\cos^2 t(4+3\sin^2 t)}{4+\sin^2 t}
    \right) \\
    &= \frac{2\left(8+8\sin^2 t
    +\sin^2 t\cos^2 t(7+3\sin^2 t)\right)}
    {(4+\sin^2 t)^2}
    >0.
    \end{align*}
    By Lemma \ref{example verify lemma}, the smooth metric measure space $(M,g,f)$ has $\ric^f_2 >0$. 
    
    Finally, if $Y,Z$ are orthonormal vectors tangent to the spherical fiber, then
    \[\ric_2\left( \partial_t, \{Y,Z\} \right) 
    = -2\frac{\phi''}{\phi}
    = \frac{2(-2+9\sin^2 t)}{4+\sin^2 t}
    \]
    which is negative when $\sin^2 t<\frac{2}{9}$.
\end{proof}